\documentclass[11pt]{article}
\usepackage[dvips]{epsfig}
\usepackage{amsmath,amssymb,euscript,graphicx,amsfonts,verbatim}
\usepackage{enumerate,color}
\usepackage{graphicx}
\usepackage[colorlinks=true,
linkcolor=blue,
urlcolor=blue,
citecolor=blue]{hyperref}

\usepackage{soul} 

\usepackage{graphicx}

\newtheorem{theorem}{Theorem}[section]
\newtheorem{proposition}[theorem]{Proposition}

\newtheorem{corollary}[theorem]{Corollary}

\newtheorem{problem}[theorem]{Problem}
\newtheorem{example}[theorem]{Example}

\newenvironment{proof}{{\noindent \sc Proof. }}{\hfill $\Qed$\\}

\newcommand{\la}{\langle}
\newcommand{\ra}{\rangle}

\newcommand{\Qed}{\rule{2.5mm}{3mm}}
\newcommand{\Aut}{\hbox{{\rm Aut}}\,}

\renewcommand{\H}{\mathcal{H}}

\newcommand{\ZZ}{\mathbb{Z}}

\newcommand{\C}{{\cal{C}}}

\newcounter{case}

\renewcommand{\thecase}{\arabic{case}}

\newcounter{subcase}

\numberwithin{subcase}{case}

\begin{document}

	
	\begin{center}
		{\bf\Large  Infinite families of vertex-transitive graphs\\  
			with prescribed Hamilton compression} \\ [+4ex]
		Klavdija Kutnar{\small$^{a,b,}$}\footnotemark  \ 
		Dragan Maru\v si\v c{\small$^{a, b, c,*,}$}\footnotemark
		\ and \\
		\addtocounter{footnote}{0} 
		Andriaherimanana Sarobidy Razafimahatratra{\small$^{a,b}$} 
	\\ [+2ex]
	{\it \small 
		$^a$University of Primorska, UP IAM, Muzejski trg 2, 6000 Koper, Slovenia\\
		$^b$University of Primorska, UP FAMNIT, Glagolja\v ska 8, 6000 Koper, Slovenia\\
		$^c$IMFM, Jadranska 19, 1000 Ljubljana, Slovenia}
\end{center}

\addtocounter{footnote}{-1}
\footnotetext{The work of Klavdija Kutnar  is supported in part by the Slovenian Research Agency (research program P1-0285 and research projects N1-0062, J1-9110, J1-9186, J1-1695, J1-1715, N1-0140, J1-2451, J1-2481, N1-0209, J3-3001).}
\addtocounter{footnote}{1}
\footnotetext{
	The work of Dragan Maru\v si\v c is supported in part by the Slovenian Research Agency (I0-0035, research program P1-0285 and research projects N1-0062, J1-9108, J1-1695, N1-0140, J1-2451, J3-3001).

	~*Corresponding author e-mail:~dragan.marusic@upr.si}

\begin{abstract}
	Given a graph $X$ with a Hamilton cycle $C$, the 
	{\em compression factor $\kappa(X,C)$ of $C$} 
	is the order of the largest cyclic subgroup 
	of $\Aut(C)\cap\Aut(X)$, and the 
	{\em Hamilton compression $\kappa(X)$ of $X$ } 
	is the maximum of $\kappa(X,C)$
	where $C$ runs over all Hamilton cycles in $X$.
	Generalizing the well-known open problem regarding the existence of vertex-transitive graphs without Hamilton paths/cycles, it was asked by Gregor, Merino and M\"utze in
	[``The Hamilton compression of highly symmetric graphs'', {\em arXiv preprint} \href{https://arxiv.org/abs/2205.08126}{arXiv: 2205.08126v1} (2022)]
	whether for every positive integer $k$ there exists
	infinitely many vertex-transitive graphs (Cayley graphs) with Hamilton compression equal to $k$.
	Since  an infinite family of Cayley graphs with Hamilton compression equal to $1$ was given there, the question is completely resolved 
	in this paper in the case of Cayley graphs with a construction of 
	Cayley graphs of semidirect products $\ZZ_p\rtimes\ZZ_k$ where $p$ is a prime and $k \geq 2$ a divisor of $p-1$.
	Further, infinite families of non-Cayley vertex-transitive graphs with Hamilton compression equal to $1$ are given.
	All of these graphs being metacirculants, some additional results on Hamilton compression of metacirculants of specific orders are also given.
\end{abstract}

\begin{quotation}
\end{quotation}

\begin{quotation}
	\noindent 
	{\em Math. Subj. Class.:} 05C25, 20B25.
\end{quotation}


\section{Introductory remarks}
\label{sec:intro}
\noindent

Following \cite{GMM22} we say that a Hamilton cycle 
$C=v_0,v_1,\ldots, v_{n-1},v_0$ in a graph $X$ is $k$-{\em symmetric} 
(and that it admits a rotational symmetry of order $k$) if
there exists an automorphism $\alpha$ of $X$ such that 
$\alpha(v_i)=v_{i+n/k}$ for each $i\in\ZZ_n$. 
This means that
$Aut(C)\cap Aut(X)$ contains a cyclic group of order $k$.
We refer to the maximum $k$ for which the Hamilton cycle $C$
of  $X$ is $k$-symmetric as the 
{\em compression factor of} $C$, denoted by $\kappa(X,C)$.

As observed in \cite{GMM22} there is a natural connection of
the compression factor to the so-called LCF notation for cubic
hamiltonian graphs, see \cite{F77}, which 
describes a cubic hamiltonian graph $X$ via 
one of its Hamilton cycles $C=v_0,v_1,\ldots, v_{n-1}$
through a sequence  $[d_0,d_1,\ldots,d_{n-1}]$,
where $d_i\colon={j-i\pmod n}$ with $v_j$ being the third neighbour
of $v_i$ (different from $v_{i\pm 1}$). Also,
$-n/2<d_i\le n/2$ and  $d_i\notin\{0,\pm 1\}$.
Frucht's suggestion \cite{F77} is to search for a Hamilton cycle $C$ in $X$
whose compression factor $\kappa(X,C)$ is as large as possible.
We define the Hamilton compression $\kappa(X)$ of $X$ to be 
$$
\kappa(X)\colon=\max\{\kappa(X,C) \mid C\textrm{ is a Hamilton cycle in} X\}.
$$
If $X$ has no Hamilton cycle, we let $\kappa(X):=0$.
Also, when $\kappa(X) =1$ we will say that $X$ has a trivial Hamilton compression.

Interestingly, the concept of rotational symmetry of Hamilton cycles
is directly linked to a widely used method for constructing Hamilton cycles
in connected vertex-transitive graphs, the so-called {\em Lifting Cycle Technique}.
This approach is based on quotienting  the graph with respect to the set of orbits of a suitable semiregular automorphism. Provided the quotient graph contains a Hamilton cycle it is sometimes possible to lift this cycle to construct a Hamilton cycle in the original graph, essentially, where this construction is made easier 
when the semiregular automorphism is of prime order.
Several partial results with regards to existence of Hamilton cycles in
connected vertex-transitive graphs have been obtained using this approach
(see \cite{A89,ALW90,AP82,DGMW98,DKM21,KW85,DM88,W82}).
Of course, the semiregular automorphism in this construction is precisely
the automorphim $\alpha$ of order $k$ above. 
To summarize, this method for constructing Hamilton cycles in connected
vertex-transitive graphs produces Hamilton cycles with non-trivial rotational symmetry, and thus allows for a more condensed description of the graph via its Hamilton compression.

Additionally, we would like to remark that asking for the existence of a semiregular automorphism in a vertex-transitive graph is in line with the conjecture that 
such automorphisms always exist in these graphs (see \cite{DM81,DM18}
for the original conjecture and the results obtained thus far).
In this sense, the problem of finding Hamilton cycles with rotational symmetry is a happy marriage of two long standing open problems  in algebraic graph theory, the Lov\'asz problem \cite{L70} on existence of Hamilton paths/cycles in vertex-transitive graphs and the above-mentioned semiregularity problem for vertex-transitive graphs.

In  \cite{GMM22}  
Hamilton compression of hypercubes, Johnson
graphs, permutahedra and Cayley graphs of abelian groups
have been either exactly determined or provided
close lower and upper bounds. Also, several intriguing 
open questions have been posed there, one of them being
(a direct quote):
``Are there infinitely many vertex-transitive graphs $X$
with $\kappa(X)=k$, for each fixed integer $k$?
A particularly relevant subclass of vertex-transitive
graphs are Cayley graphs, so we may ask the same
question about Cayley graphs \cite[p. 7]{GMM22}''.
The wording of the above question may suggest that 
a certain degree of difficulty is added by restricting oneself 
to Cayley graphs. Note, however, that Cayley graphs are conjectured \cite[Section~3.1]{MP97} to exhaust asymptotically the class of vertex-transitive graphs. It is therefore not surprising that often times  when dealing with vertex-transitive graphs, 
the difficulty lies in the ``non-Cayley'' versus ``Cayley'' 
side of the equation.
We would therefore like to make a small amend to the above problem by
shifting the emphasis somewhat with an additional ``non-Cayley'' requirement.

\begin{problem}
	\label{problem-infin}
	Given a positive integer $k$, are there infinitely many  vertex-transitive non-Cayley graphs $X$ with $\kappa(X)=k$, and 
	similarly, are there infinitely many Cayley graphs $X$ with $\kappa(X)=k$?
\end{problem}

In \cite{GMM22} an infinite family of Cayley graphs with
Hamilton compression equal to $1$ is provided.
The aim of this paper is to resolve Problem~\ref{problem-infin}
in the case of Cayley graphs for all other positive integers.

We construct in Section~\ref{sec:cay}
for each $k\ge 2$  (infinitely many) Cayley graphs 
of semidirect products $\ZZ_p\rtimes\ZZ_k$,
where $p$ is a prime and $k$ a divisor of $p-1$,
with Hamilton compression equal to $k$.
(These graphs fall into the class of the so called metacirculants,
see Section~\ref{sec:cay} for the definition).
As for non-Cayley graphs, we construct in Section~\ref{sec:non-cay}
two infinite families of non-Cayley metacirculants with Hamilton compression equal to $1$;
all of them characterized by the fact that the automorphism group 
contains a normal cyclic subgroup of prime order.
This leaves non-Cayley vertex-transitive graphs with Hamilton compression $k\ge 2$ as the only remaining case to deal with.
Furthermore, as a byproduct, 
in Section~\ref{sec:pq} some additional results on Hamilton compression of metacirculants of specific orders, notably  metacirculants of order a product of two distinct  primes,  are also given.


\section{Cayley graphs with prescribed Hamilton compression}
\label{sec:cay}
\noindent

In this section we give a construction of  infinite families of Cayley graphs with prescribed  Hamilton compression $k \geq 2$.
For this purpose we bring in a special class of vertex-transitive graphs,
the so-called metacirculants.

Let $m \ge 1$ and $n \ge 2$ be integers. 
An automorphism of a graph is called $(m,n)$-{\em semiregular} if it has 
$m$ orbits of length $n$ and no other orbit. 
We say that a graph $X$ is an $(m,n)$-{\em metacirculant} 
if there exists an $(m,n)$-semiregular automorphism $\rho$ of $X$, a rotation, together with an additional automorphism 
$\sigma$, a twisted rotation, of $X$ normalizing $\rho$, that is,
$$
\sigma\rho\sigma^{-1}=\rho^r \textrm{ for some } r\in\ZZ_n^*
$$
and cyclically permuting the orbits of $\rho$ in such a way that $\sigma^m$
fixes a vertex of $X$. For $m=1$ the graphs are circulants
and so we will from now on assume, unless specified otherwise, 
that $m\ge 2$. (Hereafter $\ZZ_n$ denotes the ring of residue classes modulo $n$ as well as the additive cyclic group of order $n$, 
depending on the context.) Note that this implies that $\sigma^m$ 
fixes a vertex in every orbit of $\rho$. To stress the role of these two automorphisms in the definition of the metacirculant $X$ we shall say 
that $X$ is an $(m,n)$-metacirculant relative to the ordered pair 
$(\rho,\sigma)$. A graph $X$ is a {\em metacirculant} if it is an $(m,n)$-metacirculant for some $m$ and $n$. This definition is equivalent with the original definition of a metacirculant by Alspach and Parsons (see \cite{AP82}). 
The abstract group isomorphic to the group $G=\la\rho,\sigma\ra$
is a semidirect product of two cyclic groups.
As a starting point for an explicit construction of such graphs we identify the set $V(m,n)$ on which a semidirect product $G=\la\rho,\sigma\ra$, of two cyclic groups, acts as  follows:

\begin{equation}
	\label{eq:vertex}
	V(m,n) =\{v_i^j\colon i\in\ZZ_m, j\in\ZZ_n\}.
\end{equation}
\noindent
Let  $r\in\ZZ_n^*$ and define the two permutations $\rho$ and $\sigma$ on $V$ by the rules

\begin{equation}
	\label{eq:rho}
	\rho \colon v_i^j\mapsto v_i^{j+1}, i\in\ZZ_m, j\in\ZZ_n,
\end{equation}

\noindent
and 

\begin{equation}
	\label{eq:sigma}
	\sigma  \colon v_i^j\mapsto v_{i+1}^{rj}, i\in\ZZ_m, j\in\ZZ_n.
\end{equation}

\medskip
\noindent
Following  \cite{AP82} we can now construct a graph admitting 
a transitive action  of $G$ on the vertex set $V$  by specifying the neighbours' set  of vertex $v_0^0$, all the remaining adjacencies  follow   from  the transitive action of the group $G$. 
Note that in \cite{AP82} various 
necessary and sufficient conditions are given for  a metacirculant to be a Cayley graph (or a non-Cayley graph). 
Two  special instances  will be of use here.
First, if we require that $\sigma$ be of order $m$, and that also $r$ is of order $m$ in $\ZZ_n^*$, then the corresponding graph is a Cayley graph \cite[Theorem 9]{AP82}. 
As we prove below, infinitely many Cayley graphs with prescribed Hamilton compression $k\geq2$ exist among these graphs. 
And second,  for $p$ and $q$ primes such that 
$p\equiv {1\pmod k}$,
a non-Cayley  $(q,p)$- metacirculant can be constructed
using \cite[Corollary 13]{AP82}, whose Hamilton compression equals 1, as we shall see in Section~\ref{sec:non-cay}

With $m, n$ and $r$ as above, we now define the graph $X(m,n;r)$ to have vertex set $V(m,n)$ and  edge set arising from the adjacencies:
$$
v_i^j \sim v_i^{j+r^i}
{\rm and}  \,\,\,
v_i^j \sim v_{i+1}^{j}
$$
\noindent
for all $i \in \ZZ_m$ and $j \in \ZZ_n$,
with the additional requirement for $r$ to have order $m$  in $\ZZ_p^*$.
Observe that   $X(m,n;r)$ has valency $4$ when
$m\ge 3$ while $X(2,n;r)$ is the $n$-prism (and so a cubic graph).

\medskip

The next proposition gives a lower bound for 
Hamilton compression of $X(m,n;r)$.

\begin{proposition}
	\label{pro:twist}
	Let $m\ge 2$ and $n$ be integers and
	$r \neq 1$ of order $m$ in $\ZZ_n^*$ 
	such that $r-1 \in \ZZ_n^*$ (is coprime with $n$).
	Then we have
	\begin{enumerate}[(i)]
		\itemsep=0pt
		\item if $m\ge 3$ then $\kappa(X(m,n;r)) \geq m$;
		\item if $n$ is odd then $\kappa(X(2,n;r)) \geq 2$;
		\item if $n$ is even then $\kappa(X(2,n;r)) \geq \frac{n}{2}$.
	\end{enumerate}
\end{proposition}

\begin{proof}
	Letting $X=X(m,n;r)$, we have 
	$V(X) = V(m,n)$  defined in (\ref{eq:vertex})
	and the automorphisms $\rho$ and $\sigma$ mapping 
	according to the rules given in 
	(\ref{eq:rho}) and (\ref{eq:sigma}).
	
	To show that $\kappa(X) \geq m$  we just need to construct a Hamilton cycle in $X$  with compression factor $m$.
	Suppose first that $m=2$.
	Then $X$ is the $n$-prism and the permutation $\tau$, 
	mapping according to the rule

	\begin{equation}
		\label{eq:tau}
		\tau  \colon v_i^j\mapsto v_{i+1}^{n-j}, i\in\ZZ_m, j\in\ZZ_n.
	\end{equation}
	
	\noindent
	is an automorphism of $X$.
	Consider the quotient $X_\tau$. It may be seen that for  
	$n$ odd we get a liftable Hamilton  cycle $X_\tau$
	and so $\kappa(X(2,n;r)) \geq 2$.
	For $n$ even, however,  we cannot find such a cycle and so in this case
	we need to  consider the quotient $X_{\rho^2}$. It may be seen that
	a liftable Hamilton cycle exist in this case so that 
	$\kappa(X(2,n;r)) \geq \frac{n}{2}$ in this case.
	
	Suppose now that $m\ge 3$.
	Then a Hamilton cycle in $X$ is obtained by taking the quotient
	$X_{\sigma}$ of $X$ with respect to the automorphism $\sigma$
	whose orbits are
	$$
	V_j = \{v_0^{j},v_1^{rj}, \cdots, v_{m-1}^{r^{m-1}j} \} =
	\{v_i^{r^ij}\mid i\in\ZZ_m\},\ j\in\ZZ_n.
	$$
	Since each  for each $j \in \ZZ_n$, the vertex $v_0^j$ is adjacent to $v_0^{j+1}$ we  have that
	$V_j \sim V_{j+1}$  for each $j \in \ZZ_n$, 
	giving rise to a  Hamilton cycle
	$$
	C = V_0,V_1,\ldots,V_{n-1}, V_0
	$$ 
	in $X_{\sigma}$.
	As it turns out, $C$ lifts to a full Hamilton cycle
	in $X$. For this purpose we first show  that at least one of the edges $V_jV_{j+1}$, $j\in\ZZ_n$, is a multiple edge.
	In fact, as  we prove below, exactly two edges in $C$ are double edges, sufficing as we shall see, to lift $C$ to a Hamilton cycle in $X$.
	
	Consider the vertex $v_0^j\in V_j$. It has a neighbour $v_0^{j+1}$ in $V_{j+1}$. Furthermore,  $v_1^j$ and $v_{m-1}^j$ are also neighbours of $v_0^j$. But  $V_{j+1}$ contains the vertex
	$v_1^{r(j+1)}$ (with subscript $1$)
	and the vertex  $v_{m-1}^{r^{m-1}(j+1)}$ (with subscript $m-1$).
	If either of these two pairs of vertices coincide, that is, if
	$v_1^j = v_1^{r(j+1)}$ or  if
	$v_{{m-1}}^j = v_{m-1}^{r^{m-1}(j+1)}$, then $V_jV_{j+1}$ is a double edge
	in $C$. Therefore, we need to have either 
	$$
	j = r(j+1) \textrm{ or }  j = r^{m-1}(j+1).
	$$
	The two respective solutions are $j =  r/(1-r)$ and 
	$j =  1/(r-1)$. 
	(Recall that by assumption, $r-1\in\ZZ_n^*$.)
	Hence, for these two values of $j$ we have a double edge between the orbits $V_j$ and $V_{j+1}$.  
	To be more precise, in the first case for example, it follows  that  both $v_0^{j+1}$ and $v_1^{r(j+1)} = \sigma(v_0^{j+1})$ are  neighbours of $v_0^j$. This means that ``the voltages'' between $V_j$ and $V_{j+1}$
	differ by $1$ (a generator of $\ZZ_m$) and so $C$ 
	lifts to a full Hamilton cycle in $X$, as required.
\end{proof}

We adopt the following notation.
Given an $(m,n)$-metacirculant $X$
with an $(m,n)$-semiregular automorphism $\rho$ we 
let  $\tilde X(\rho)$ denote the subgraph obtained from $X$ 
by removing all the edges joining two vertices from the same orbit of 
$\rho$.

We  give below an infinite family of Cayley graphs
with a prescribed Hamilton compression $k\geq2$.

\begin{theorem}
	\label{the:family}
	Let $k\geq2$ be a positive integer and $p$ a prime such that
	$p\equiv {1\pmod k}$
	and let $r \in \ZZ_p^*$ have order $k$. 
	Then $\kappa(X(k,p;r)) = k$.
\end{theorem}

\begin{proof}
	Let $X=X(p,k;r)$. Then letting $n=p$ and $m=k$ we have that
	$\kappa(X(k,p;r))\geq k$ by Proposition~\ref{pro:twist}.
	(The additional assumption on $r-1$ being coprime with $n$ is clearly satisfied too.)
	Suppose that $\kappa(X)  > k$. Then since $|V(X)| = kp$ and $k < p$, 
	there exists a Hamilton cycle $C$ in $X$ with compression factor
	a multiple of $p$. In particular, this implies the existence of an automorphism $\pi \in \Aut(C)$ of order $p$ such that 
	the subgraph $\tilde X(\pi)$
	is a connected  (and hence a hamiltonian) graph.    
	Note that $\rho$ is an element of order $p$ which  does not 
	have this property. Namely, by definition $\tilde X(\rho)$ is the disconnected graph $pX_k$. In particular this implies that not all of the automorphisms of order $p$ in $\Aut(X)$ are conjugate. But this can only happen if a Sylow $p$-subgroup of $\Aut(X)$  is of order at least $p^2$. 
	Now since $k < p$, a Sylow $p$-subgroup $P$ containing $\rho$ must fix the orbits of $\rho$. But these orbits induce cycles of length $p$, while the bipartite graphs between two  neighbouring orbits of $\rho$ induce  perfect matchings. Consequently, $P$ coincides with $\la \rho \ra$, 
	contradicting the existence of the automorphism $\pi$. 
	This  shows that $\kappa(X) = k$, 
	as required. 
\end{proof}

In order to prove that there are infinitely many Cayley graphs with prescribed Hamilton compression we use the well-known Dirichlet's theorem on arithmetic progressions, the so-called
Dirichlet prime number theorem. 

%
%
\begin{proposition}\label{pro:dirichlet}
	{\rm \cite{D37} {\bf (Dirichlet prime number theorem)}}
	If $a$ and $b$ are relatively prime positive integers, then there are infinitely many primes of the form $a + jb$ with $j \in\ZZ$.
\end{proposition}

%
%
\begin{corollary}
	\label{cor:inf}
	Given a positive integer $k\geq2$ there exist infinitely many Cayley graphs with Hamilton compression equal to $k$. 
\end{corollary}

\begin{proof}
	By Theorem~\ref{the:family}
	the graph $X(k,p;r)$ where $p$ is a prime, $k$ divides $p-1$
	and $r \in \ZZ_p^*$ is of order $k$,
	has Hamilton compression equal to $k$. On the other hand, 
	by Proposition~\ref{pro:dirichlet}, letting $a=1$ and $b=k$
	there exist infinitely many primes $p$ such that
	$k$ divides $p-1$. 
\end{proof}


\section{Non-Cayley graphs with trivial Hamilton compression}
\label{sec:non-cay}
\noindent

In this section we give  infinitely many  non-Cayley graphs 
with Hamilton compression equal to $1$.
For this purpose we bring in the construction of non-Cayley metacirculants
of order a product of two distinct primes, given in  \cite{AP82}.
First, given a prime $q$, we construct infinitely many  non-Cayley $(q,p)$-metacirculants, $p$ a prime, using Proposition~\ref{pro:dirichlet}. 
Namely, there  are infinitely
many primes $p$ in the arithmetic progression
$1+q^2, 1+2q^2, 1+3q^2, \cdots$. 
So let $p = 1+ N q^t$ be such a prime  with $t\geq2$ and  $(N,q) =1$. Then letting $\lambda$ be a generator of $\ZZ_p^*$, we set $r = \lambda^N$, so that the order of $r \in \ZZ_p^*$ is $q^t$. Further, let 
$R = \{ \la r^q \ra \cup -\la r^q \ra \}$.  
Then  a non-Cayley  $(q,p)$-metacirculant  $Y(q,p)$ is defined to have 
vertex set  $V(q,p)$  (recall (\ref{eq:vertex}))
and   edge set

$$
E(Y(q,p)) = \{ v_i^j v_{i+1}^j \mid i \in \ZZ_q, j \in \ZZ_p\} \cup
\{ v_i^j v_i^{j+x} \mid i \in \ZZ_q, j \in \ZZ_p, x \in r^i R\}. 
$$

\smallskip
\noindent
This construction of Alspach and Parsons can be slightly altered 
if one wants to get  non-Cayley metacirculants with smaller valency.
By letting $Q  = \{ \la r^{q^{t-1}} \ra \cup -\la r^{q^{t-1}} \ra \}$,
and defining the graph $Z(q,p)$ to have  the same vertex set as $Y(q,p)$,
while the edge set is

$$
E(Z(q,p)) = \{ v_i^j v_{i+1}^j \mid i \in \ZZ_q, j \in \ZZ_p\} \cup
\{ v_i^j v_i^{j+x} \mid i \in \ZZ_q, j \in \ZZ_p, x \in r^i Q\}. 
$$

\smallskip
\noindent
Observe that $Y(q,p) =Z(q,p)$ for $t=2$. 
In the next theorem we show that both 
$Y(q,p)$ and $Z(q,p)$ have trivial Hamilton compression.

%
%
\begin{theorem}
	\label{the:YZ}
	Let $p$ and $q$ be primes such that 
	$p\equiv {1\pmod {q^2}}$
	and let $r \in \ZZ_p^*$ 
	have order $q^t$,  with $t \geq2$. 
	Then $\kappa(Y(q,p)) = \kappa(Z(q,p)) =1$.
\end{theorem}

\begin{proof}
	We simplify the notation by letting $Y = Y(q,p)$, $Z = Z(q,p)$ 
	and  $V =V(q,p)$.
	Now, observe that the two permutations $\rho$ and $\sigma$, given in  \ref{eq:rho}) and (\ref{eq:sigma}
	with $m=q$ and $n=p$, are automorphisms of both $Y$  and  $Z$.
	While the order of $\rho$ is clearly $p$, the order of $\sigma$ is  $q^t$
	and  $q^2$, respectively, in the graphs $Y$ and  $Z$.
	Since $t\geq2$, the group $\la  \rho, \sigma \ra$
	acts transitively but not regularly on $V(q,p)$ 
	and so both $Y$ and $Z$ are non-Cayley graphs.
	
	By the construction it is therefore clear that semiregular automorphisms of 
	$Y$ and $Z$ must be of order $p$. 
	Moreover, since $qp < p^2$, the orbits of a 
	Sylow $p$-subgroup $P$ containing  $\la \rho \ra$  
	coincide with the orbits of $\la \rho \ra$ and consequently 
	$P = \la \rho \ra$.
	In fact, $P$ is a normal subgroup. Namely, we know by \cite{DM94} that the only non-Cayley $(q,p)$-metacirculants whose automorphism group is not $p$-imprimitive are the Petersen graph, its complement, 
	and three basic orbital graphs of order $57$ obtained from the action of $\operatorname{PSL}(2,19)$ on cosets of $A_5$ of valencies $6$, $20$ and $30$, and their complements, of course.  Using a computer search,  the compression number of the basic orbital graphs (and their complements) arising from the action of $\operatorname{PSL}(2,19)$ on cosets of $A_5$ can be verified to be equal to $19$. Moreover, the Petersen graph does not have a Hamilton cycle, so its compression number is equal to $0$ by definition. Its complement, on the contrary, is hamiltonian and has compression number equal to  $5$.
	This means that the automorphism 
	groups of our graphs $\Aut (Y)$ and $\Aut (Z)$
	of our graphs $Y$ and $Z$ (being $4$-valent) 
	are not primitive and so must be  $p$-imprimitive. Consequently,  
	$P = \la \rho \ra$ is normal in $\Aut (Y)$ and $\Aut (Z)$.
	(Let us also mention that the Petersen graph is the only generalized Petersen graph  $GP(p,r)$ with a primitive automorphism group and hence  with a Sylow $p$-subgroup not normal.)
	But the orbits of $\rho$ are not independent sets and so the subgraphs  $\tilde Y(\rho)$ and $\tilde Z(\rho)$ are unions of $q$-cycles 
	for $q\geq3$ and perfect matchings $pK_2$ for $q=2$ and thus disconnected. It follows that no semiregular automorphism of order $p$  gives rise to a rotational symmetry of a Hamilton cycle. Therefore, the corresponding Hamilton compressions of $Y$ and $Z$ are trivial.
\end{proof}

It is worthwhile to mention that  in \cite{AP82,AP82-2,DM88}, Hamilton cycles with trivial symmetries were constructed for metacirculant graphs.

The next corollary is now immediate.

%
%
\begin{corollary}
	\label{cor:infnon}
	There exist infinitely many non-Cayley graphs with Hamilton compression equal to $1$. 
\end{corollary}


\section{Vertex-transitive graphs of specific orders}
\label{sec:pq}
\noindent

The existence of Hamilton cycles in connected vertex-transitive graphs 
has been proved for various special orders (see \cite{}), such as for example:
($p$ denotes a prime): 

\begin{enumerate}[(i)]
	\itemsep=0pt
	\item $p$, $p^2$, $2p^2$, $p^3$, where $p$ is a prime \cite{DM87}, 
	\item product of two distinct prime numbers 
	(except for the Petersen graph) \cite{DKM21}, and
	\item $4p$, $6p$ and $10p$ (except for the Coxeter graph and some families of order $6p$ and $10p$) \cite{KM08,KMZ12,KS09}; 
\end{enumerate}

It is worth mentioning that in some of the above constructions, the lifting cycle technique was used, giving us, as a minimum,  a lower bound for the Hamilton compressions of the corresponding Hamilton cycles. 
On the other hand, the primary goal being simply the construction of a Hamilton cycle rather than one with largest possible rotational symmetry,  
it is not surprising that in some cases the Hamilton cycles have  trivial compression factors.

\begin{example}
	\label{ex:}
\end{example}

Recall that vertex-transitive graphs of prime order are necessarily circulants, and so the corresponding Hamilton compression 
coincides with the order of the graph.
Also, since every transitive group of prime square degree is regular, 
it follows that a vertex-transitive graph of order $p^2$, $p$ a prime, 
is  a Cayley graph of either $\ZZ_{p^2}$ or $\ZZ_{p}^{2}$.
In the first case the corresponding Hamilton compression again 
coincides with the order $p^2$, while in the second case,  it  
equals $p$ (see \cite{DM85}).
As for vertex-transitive graphs of order $p^3$, it is known that they are always{  Cayley graphs \cite{DM85}}. Also, they have been proved to be hamiltonian \cite{DM85}.
In the case of Cayley graphs of abelian groups 
$\ZZ_{p^3}$, $\ZZ_{p^2} \times \ZZ_p$ or $\ZZ_{p}^{3}$
of order $p^3$ we may argue as above that their Hamilton
compressions are respectively $p^3$, $p^2$ and $p$.
As for the two nonisomorphic non-abelian groups of order $p^3$,
$p$-symmetric Hamilton cycles  were constructed in  
\cite[Theorem 4.4]{DM85}.
While the proof in \cite{DM85} is quite technical it may be simplified by applying  the following beautiful result about cycle spaces in Cayley graphs of abelian groups proved a few years later  \cite{ALW90}.

%
%

\begin{proposition}
	\label{pro:alspach}
	{\rm [Theorem 2.1.]}\cite{ALW90}
	Let $X$ be a connected Cayley graph on a finite abelian group $G$,
	and let $\C$  and $\H$  be the cycle space and the Hamilton space of $X$ respectively. Then:
	\begin{enumerate}[(i)]
		\itemsep=0pt
		\item $\H = \C$ when $X$ is either bipartite or has odd order;
		\item $\H $ has co-dimension $2$ in $\C$ when $X$ is a prism over 
		an odd length cycle;
		\item $\H$  has co-dimension 1 in $\C$ in all other situations.
	\end{enumerate}
\end{proposition}

%
%

\begin{proposition}
	\label{pro:DM}
	{\rm [Theorem 4.4]}\cite{DM85}
	Let $X$ be a Cayley graph of a non-abelian group $G$ of order $p^3$, 
	$p$ a prime. Then $\kappa(X) \geq p$.
\end{proposition}

\begin{proof}
	We may assume that $X$ contains as a subgraph the graph 
	$Y = Cay (G,\{a,b\}$ where $a$ and $b$ are noncommuting elements 
	of $G$ and its commutator $c =[a,b]$ is a central element. In fact,
	$[G,G] =Z(G)$.
	Since $\langle c\rangle$ is generated by $a$ and $b$, the subgraph $\tilde X(c)$
	is connected and so the quotient $X_{c}$, a Cayley graph of an abelian  group of odd order $p^2$, lifts to a connected graph. In particular there must be a cycle $C$  in $X_{c}$ which lifts to a full cycle of length $p|C|$ in $X$. Besides by Proposition~\ref{pro:alspach} the corresponding cycle and Hamilton spaces coincide and so $C$ must be a sum of Hamilton cycles
	in $X_{c}$. It follows that at least one of these Hamilton cycles lifts to a (full) Hamilton cycle in $X$, as required. 
\end{proof}

This naturally leads us to  the product of two distinct primes as a next 
step in revisiting the above  constructions of Hamilton cycles.
Following \cite[Theorem~2.1]{DM94}, vertex-transitive graphs of order $pq$, $p > q$ primes, fall into three classes, depending on 
(im)primitivity of their automorphism group.
While one class  is made of graphs characterized by the fact that the full automorphism group as well as all of its transitive subgroups are primitive, the other two classes  are made of  graphs  with automorphism groups containing a transitive subgroup acting imprimitively on the set of vertices.
The first class of these two classes consists of $(q,p)$-metacirculants,  where the imprimitive subgroup  is  the normalizer of a Sylow $p$-subgroup 
and the blocks coincide with the orbits of this Sylow $p$-subgroup.
The graphs belonging to second class only exist when both $p$ and $q$ are Fermat primes. Here the  automorphism group contains an imprimitive subgroup of blocks of size $q$ but no such subgroup with blocks of size $p$. 

Hereafter, we deal with the class of metacirculants of order $pq$
and leave the other two classes to the sequel of this paper.
It follows from \cite{GMM22}  that a Cayley graph of an abelian group  has trivial Hamilton compression if and only if it is (a circulant) of odd square free order with a canonical generating set.  
We therefore have that  the Hamilton compression of  
a circulant ${\rm Cay} (\ZZ_{pq},S)$  is 
$pq$ when $S$ contains a generator of $\ZZ_{pq}$ and is $1$ 
otherwise. We may therefore restrict ourselves to
$(q,p)$-metacirculants arising from a transitive action of a non-trivial
{(i.e., non-isomorphic to the direct product)}
semidirect product  $\ZZ_p\rtimes\ZZ_m$,
where $m$ can be chosen to be a power of $q$.

For such a $(q,p)$-metacirculant $X$, let $\rho$ be a corresponding $(q,p)$-semiregular automorphism  (with a transitive normalizer) and let $P$ be a Sylow $p$-subgroup $P$ containing $\rho$.
Since $q<p$ the orbits of $P$ coincide with the orbits of $\rho$
(and of any $(q,p)$-semiregular element of $P$). We may therefore define
the graph $\tilde X(P)$ to be the graph  $\tilde X(\rho)$.
Also, since any two Sylow $p$-subgroups are conjugate
the graph $\tilde X(P)$ is independent of  the choice of a Sylow $p$-subgroup, and we may therefore 
simplify the notation to $\tilde X$. 

We are now ready to characterize Hamilton compressions of $(q,p)$-metacirculants for $q < p$ primes.

%
%

\begin{theorem}
	\label{the:metapq}
	Let $X$ be a $(q,p)$-metacirculant, where $q < p$ are primes.
	Then  the Hamilton compression of $X$ is as follows:
	\begin{enumerate}[(i)]
		\itemsep=0pt
		\item $\kappa(X) = 0$ if $X$ is isomorphic to the Petersen graph;
		\item $\kappa(X) = 1$ if $X$ is a non-Cayley graph and $Y_{X}$ is a disconnected graph;
		\item $\kappa(X) = q$ if $X$ is a Cayley graph and $Y_{X}$ is a disconnected graph; 
		\item $\kappa(X) = p$ if $X$ is not a Cayley graph of both a dihedral and a cyclic group of order $2p$ and $Y_{X}$ is a connected graph;
		\item $\kappa(X) = 2p$ if $X$ is a Cayley graph of both a dihedral and a cyclic group of order $2p$ and $Y_{X}$ is a connected graph.
	\end{enumerate}
\end{theorem}

\begin{proof}
	Identifying $V(X)$ with $V$ from \eqref{eq:vertex} we have that  the two permutations $\rho$ and $\sigma$, given in  \eqref{eq:rho} and \eqref{eq:sigma} with $m=q$ and $n=p$, respectively, 
	give rise to the semidirect product acting transitively 
	on the vertex set $V(X)$.
	We know that, with the exception of the Petersen graph, 
	$(q,p)$-metacirculants are hamiltonian 
	(see \cite{AP82,AP82-2,DM88})
	and so $\kappa(X) \geq 1$. 
	We distinguish two cases depending 
	on connectedness of $\tilde X$. Let $P$ be a Sylow $p$-subgroup
	containing $\rho$.

	\smallskip
	{\sc Case 1.} $\tilde X$ is a disconnected graph.
	
	\smallskip
	Then clearly  $\kappa(X) \neq p,pq$.
	Moreover, the orbits of $P$ are not independent sets.
	Let us first deal with the case $q=2$. Then  $\tilde X$ is a perfect matching.  Moreover, if $X$ is a non-Cayley graph then we must have  
	$p\equiv {1\pmod 4}$, in which case $\Aut (X)$ contains  no semiregular involution and so $\kappa(X) =1$. If, on the other hand, $X$ is a Cayley graph then it is a Cayley graph of a dihedral group $D_{2p}$, and so it contains the $p$-prism as a subgraph and so  
	$\kappa(X) =2$.
	
	Suppose now that $q$ is odd.
	Since the orbits of $P$ are not independent sets,
	$X$ must necessarily contain a subgraph isomorphic to
	$X(q,p,r)$ for an appropriate $r \in \ZZ_p^*$, 
	when $X$ is a Cayley graph. By  Theorem~\ref{the:family} we have that $\kappa(X) = q$. When $X$ is a non-Cayley graph, then it admits no  semiregular automorphism  of order $q$  and so $\kappa(X) = 1$.
	
	\smallskip
	{\sc Case 2.} $\tilde X$ is a connected graph.
	
	\smallskip
	Suppose  first that $q=2$. Then $X$  can be represented by a triple
	of the form $[S,S',T]$ which defines the adjacencies  in $X$.
	The sets $S$ and $S'$ are subsets of $\ZZ_p^*$,
	and $T$ is a subset of $\ZZ_p$, and we have that
	the neighbours of $v_0^j$ are all vertices of the 
	form $v_0^{j+s}$, $s \in S$
	and all vertices of the form $v_1^{j+t}$, $t \in T$, and similarly
	the neighbours of $v_1^j$ are all vertices of the form 
	$v_1^{j+s'}$, $s' \in S'$,
	and all vertices of the form $v_0^{j-t}$, $t \in T$.
	It follows that $X$ is simultaneously a Cayley graph of $D_{2p}$ and $\ZZ_{2p}$ precisely when $S'=S$ and $T =-T$ can be chosen as a symmetric set. In this case, since $\tilde X$ is connected,
	a $2p$-symmetric Hamilton cycle can be constructed solely from the edges in $\tilde X$. Hence $\kappa(X) =2p$ in this case.
	No such cycle can be constructed when
	$X$ is either a non-Cayley graph or a Cayley graph of $D_{2p}$ but not of $\ZZ_{2p}$, and so $\kappa(X) = p$.
	
	Suppose now that $q$ is odd.
	Consider the quotient $X_{\rho}$. Being a circulant of order $q$, it follows by Proposition~\ref{pro:alspach} that its cycle  and Hamilton spaces coincide.
	Consequently, every cycle in $X_{\rho}$ is a sum of Hamilton cycles. Now, the fact that $\tilde X$ is connected implies that there must exist a cycle say $C$ in $X_{\rho}$ with  non-zero voltage in $\ZZ_p$. Writing $C$ as a  sum of Hamilton cycles it is clear that at least one of these cycles must have a non-zero voltage, and this cycle lifts to a full Hamilton cycle in $X$, a Hamilton cycle that admits a $p$-fold rotational symmetry. Hence $\kappa(X) =p$, as required.
\end{proof}

In the computations of compression factors of various Hamilton cycles in metacirculants, covered in this paper, one cannot fail to notice that
some of these graphs admit Hamilton cycles with different compression factors. This suggests the following definitions. Let the 
{\em Semiregularity Array} ${\rm Sem}(X)$ of a 
(vertex-transitive) graph $X$ be the 
array  $[n_1,n_2,\dots n_r]$ of positive integers arranged in ascending order with the property that $X$ admits a semiregular automorphism of order $n_i$ for each $i\in  \{1,2,\dots,r\}$.
Similarly, the {\em Hamilton Compression Array} ${\rm Ham}(X)$ of a 
(vertex-transitive) graph $X$ is the 
array  $[k_1,k_2,\dots k_t]$ of positive integers in ascending order
such that  for every $i\in  \{1,2,\dots,t\}$, we have that $X$ admits a Hamilton cycle $H_i$ with compression factor $k_i$.

Clearly ${\rm Ham}(X)$ is a subarray of ${\rm Sem}(X)$
whenever $X$ is hamiltonian.
For example  for the Petersen graph $GP(5,2)$ we have that
${\rm Sem}(X) =[1,5]$, while 
${\rm Ham}(X) = [0]$. On the other hand  the two arrays coincide 
for the complement of the Petersen graph.
It seems that vertex-transitive graphs for which the two arrays coincide deserve  special interest. We will call them \emph{ubiquitously compressible}.
While every non-Cayley $(q,p)$-metacirculant
$X$ for which the subgraph $\tilde X$ is connected is ubiquitously compressible, 
the Cayley graphs $X(m,n,r)$ are not.
We suggest the following problem.

\begin{problem}
	\label{prob:}
	\begin{enumerate}[(a)]
		\item Find  new families of ubiquitously compressible vertex-transitive graphs.
		\item Is there an insightful structural result about ubiquitously compressible
		vertex-transitive graphs?  
	\end{enumerate}
	
\end{problem}

\end{document}